\documentclass{birkau}

\usepackage{amsmath,amssymb,url}
\usepackage{enumitem}
\usepackage{amsthm,amsfonts}
\usepackage[mathscr]{euscript}
\usepackage{tikz-cd}

\numberwithin{equation}{section}

\theoremstyle{plain}
\newtheorem{theorem}{Theorem}[section]
\newtheorem{lemma}[theorem]{Lemma}
\newtheorem{proposition}[theorem]{Proposition}

\theoremstyle{definition}
\newtheorem{definition}[theorem]{Definition}
\newtheorem{example}[theorem]{Example}
\newtheorem{remark}[theorem]{Remark}

\DeclareMathOperator{\Con}{Con}
\DeclareMathOperator{\Epi}{Epi}
\DeclareMathOperator{\Hom}{Hom}
\DeclareMathOperator{\id}{id}
\DeclareMathOperator{\Obj}{Obj}

\DeclareMathOperator{\Quo}{Quo}
\DeclareMathOperator{\Eq}{Eq}

\newcommand{\C}{\mathbf{C}}
\newcommand{\D}{\mathbf{D}}
\newcommand{\M}{\mathcal{M}}
\newcommand{\N}{\mathcal{N}}
\newcommand{\F}{\mathcal{F}}
\newcommand{\Ll}{\mathcal{L}}
\newcommand{\G}{\mathcal{G}}
\newcommand{\LStr}{\Ll\mathbf{Str}}
\newcommand{\pSet}{\mathbf{pSet}}
\newcommand{\Q}{\mathbb{Q}}
\newcommand{\Ring}{\mathbf{Ring}}
\newcommand{\Set}{\mathbf{Set}}
\newcommand{\T}{\mathcal{T}}
\newcommand{\TMdl}{\T\mathbf{Mdl}}
\newcommand{\Z}{\mathbb{Z}}

\begin{document}

\title[Functor-quotients and isomorphism theorems]{On functor-quotients and their isomorphism theorems}

\author[J. M. Barrett]{Jordan Mitchell Barrett}
\address{School of Mathematics and Statistics\\Victoria University of Wellington\\Wellington 6140\\New Zealand}
\urladdr{http://jmbarrett.nz}
\email{math@jmbarrett.nz}

\author[V. Vito]{Valentino Vito}
\address{Department of Mathematics\\Universitas Indonesia\\Depok 16424\\Indonesia}
\email{valentino.vito@sci.ui.ac.id}



\begin{abstract}
The notion of a categorical quotient can be generalized since its standard categorical concept does not recover the expected quotients in certain categories. We present a more general formulation in the form of $\mathcal{F}$-quotients in a category $\mathbf{C}$, which are relativized to a faithful functor $\mathcal{F}\colon \mathbf{C} \to \mathbf{D}$. The isomorphism theorems of universal algebras generalize to this setting, and we additionally find important links between $\mathcal{F}$-quotients in the concrete category of first-order structures, and quotients defined for model-theoretic equivalence classes. By first working in this categorical setting, some quotient-related results for first-order structures can be naturally obtained. In particular, we are able to prove some isomorphism theorems in the context of model theory directly from their corresponding categorical isomorphism theorems.
\end{abstract}

\maketitle

\section{Introduction}

The concept of quotients in various structures have been studied extensively, and isomorphism theorems related to them have been researched in the past. Relevant research on hyperalgebras, for example, include studies on the isomorphism theorems of hyperrings~\cite{davvazring}, polygroups~\cite{davvazgroup}, hypermodules~\cite{zhan} and universal hyperalgebras~\cite{ebrahimi}. Some discussion on isomorphism theorems in other structures can also be found in~\cite{holcapekiso,iampan}.

General applications of category theory to universal algebra have been discussed in~\cite{borceux}. Also, Mousavi quite recently worked on free hypermodules using categorical techniques~\cite{mousavi}. Inspired by these approaches, we study isomorphism theorems in a more general setting by obtaining corresponding results from universal algebra~\cite{burris,cohn} derived from quotients defined by a faithful functor.

Let us fix the categorical notations and conventions that we will use in this paper. Abstract categories will generally be denoted by boldface, upright, capital letters --- $\C$, $\D$, etc. Special named categories (e.g., $\Set$, $\Ring$) are referred to in the same style. As per usual, a category $\C$ consists of a class of objects $\Obj(\C)$, whose members are denoted by italic capital letters $A, B, C$, etc, and a class of morphisms $\Hom(A,B)$ for every $A,B \in \Obj(\C)$, whose members are denoted by lowercase letters $f, g, h$, etc. We use the notation $\Hom(A,-)$ to denote the class of all morphisms with source $A$. Functors are denoted by the calligraphic capital letter $\F$.

First, recall that quotients in a category $\C$ are the dual of subobjects. Let $A \in \Obj(\C)$ and $\Epi(A,-)=\{\, f \in \Hom(A,-) \mid f \text{ is an epimorphism} \,\}$. We define a relation $\leq$ on $\Epi(A,-)$ such that $(f\colon A \to B)\ \leq\ (g\colon A \to C)$ if and only if $g$ factors through $f$ (i.e., there exists a morphism $h\colon B \to C$ such that $g=h \circ f$). The relation $\leq$ is easily seen to be a preorder, and thus it generates an equivalence relation on $\Epi(A,-)$ in the usual manner.

\begin{definition}[{\cite[p.~126]{mac lane}}]\label{def:quot-cat}
The \emph{categorical quotients} of $A \in \Obj(\C)$ are the equivalence classes of epimorphisms in $\Epi(A,-)$ under the equivalence relation $f \sim g \iff f \leq g$ and $g \leq f$.
\end{definition}

Here, categorical quotients are defined as (equivalence classes of) morphisms with source $A$, rather than objects obtained from $A$ via some quotient map. The object formulation of quotients is obtained by considering the target $B$ of an epimorphism $f\colon A \to B$. If $f\colon A \to B$ and $g\colon A \to C$ are such that $f \sim g$, then it follows directly from the definition of $\sim$ that $B$ and $C$ are isomorphic. Therefore, for any $\sim$-equivalence class $[f]_\sim$, the class $\{\, B \in \Obj(\C) \mid (g\colon A \to B) \in [f]_\sim \,\}$ is an isomorphism class of objects, which we identify with a quotient $[f]_\sim$ of $A$.

One issue with this definition is that it does not give a representative view of what quotients are in some common categories. For example, in $\Ring$, the inclusion map $i\colon\Z \to \Q$ is an epimorphism, despite not being surjective on the underlying sets. As a consequence, its equivalence class forms a categorical quotient, even though it is absurd to consider $\Q$ as a quotient ring of $\Z$.

In this paper, we define $\F$-quotients, a more general categorical definition of quotients which allows us to better capture the notion in algebraic categories such as $\Ring$. We devote Section~\ref{sec:F-quot} to developing the concept of an $\F$-quotient, and giving connections to free objects and the correspondence theorem of universal algebra. In Section~\ref{sec:elem-classes}, we shall see that the natural model-theoretic definition of quotients are essentially $\F$-quotients in the concrete category of first-order structures. Finally, using the theory developed in Sections~\ref{sec:F-quot} and~\ref{sec:elem-classes}, we give relatively simple proofs for the isomorphism theorems in the realm of first-order structures.

\section{On $\F$-quotients}\label{sec:F-quot}

Let $\F\colon \C \to \D$ be a faithful functor between categories. A morphism $f$ in $\C$ is an \emph{$\F$-epimorphism} if $\F(f)$ is an epimorphism in $\D$. Since $\F$ is faithful, any $\F$-epimorphism is necessarily an epimorphism.

\begin{definition}\label{def:F-quot}
Let $\F\colon \C \to \D$ be a faithful functor. Let $A \in \Obj(\C)$ and let
\[\Epi_\F(A,-) = \{\, f \in \Hom(A,-) \mid f \text{ is an } \F\text{-epimorphism} \,\}.\]
We define a relation $\leq$ on $\Epi_\F(A,-)$ by
\[(f\colon A \to B)\ \leq\ (g\colon A \to C)\ \iff\ g=h \circ f \text{ for some } h\colon B \to C.\]
Then, \emph{$\F$-quotients} of $A$ are equivalence classes under the equivalence relation $f \sim g \iff f \leq g$ and $g \leq f$. The class of all $\F$-quotients of A is denoted by $\Quo_\F(A)$.
\end{definition}

Since $\F$-epimorphisms are epimorphisms, thus $\Epi_\F(A,-) \subseteq \Epi(A,-)$. Therefore, the preceding definition allows us to restrict the set $\Epi(A,-)$ to a certain extent. We note that Definition~\ref{def:F-quot} generalizes the usual categorical definition of a quotient given in Definition~\ref{def:quot-cat}.

\begin{example}
Let $\D = \C$, and $\F\colon \C \to \C$ be the identity functor. The $\F$-quotients of $A \in \Obj(\C)$ are exactly the categorical quotients of $A$.
\end{example}

\begin{example}
Let $R$ be a unital ring and let $\F\colon\Ring \to \Set$ be the forgetful functor. Then $\Epi_\F(R,-)$ precisely contains surjective ring homomorphisms with $R$ as the source, in contrast to the set $\Epi(R,-)$ of all ring epimorphisms from $R$, which contains the troublesome inclusion map $i\colon\Z \to \Q$.
\end{example}

Note that the class $\Quo_\F(A)$ might be proper for large categories, although in the next section, we show that it is a set in the concrete category of first-order structures. The relation $\leq$ forms a well-defined partial order on $\Quo_\F(A)$, where for every $[f]_\sim,[g]_\sim \in \Quo_\F(A)$, we define $[f]_\sim \leq [g]_\sim \iff f\leq g$.

\begin{proposition}\label{prop:g/f-epi}
Let $\F\colon\C \to \D$ be a faithful functor and let $A \in \Obj(\C)$. If $f,g \in \Epi_\F(A,-)$ and $f \leq g$, then the induced morphism $h$ such that $g=h \circ f$ is unique. Moreover, $h$ is an $\F$-epimorphism.
\end{proposition}
\begin{proof}
Suppose $h_1,h_2$ are such that $h_1 \circ f=g=h_2 \circ f$, then clearly $\F(h_1) \circ \F(f)=\F(h_2) \circ \F(f)$, and so since $\F(f)$ is an epimorphism, we have $\F(h_1)=\F(h_2)$. Since $\F$ is faithful, we have $h_1=h_2$. To prove that $\F(h)$ is an epimorphism, let $j_1,j_2$ be such that $j_1 \circ h=j_2 \circ h$. We thus have $j_1 \circ g=j_2 \circ g$, and so $j_1=j_2$ since $g$ is an epimorphism.
\end{proof}

Let $f,g \in \Epi_\F(A,-)$ and $f \leq g$. We will denote the unique induced morphism $h$ from Definition~\ref{def:F-quot} by the suggestive notation $g/f$. Moreover, if $f\colon A \to B$, we will denote $B$ as $A/f$. By Proposition~\ref{prop:g/f-epi}, we have $g/f \in \Epi_\F(A/f,-)$. We note that $f/f$ is the identity morphism and that if $g \leq h$, then $(h/g) \circ (g/f) = (h/f)$.

\begin{proposition}[Categorical Third Isomorphism Theorem]\label{prop:cat-third-iso}
Let $\F\colon\C \to \D$ be a faithful functor and let $A \in \Obj(\C)$. If $f,g \in \Epi_\F(A,-)$ and $f \leq g$, then
\[\frac{(A/f)}{(g/f)}=\frac{A}{g}.\]
\end{proposition}
\begin{proof}
The following diagram illustrates the proof:
\[\begin{tikzcd}
A \arrow[d,"f"'] \arrow[r, "g"] & A/g=\frac{(A/f)}{(g/f)} \\
A/f \arrow[ur,"g/f"']
\end{tikzcd}\]
\end{proof}

\begin{lemma}
Let $\F\colon\C \to \D$ be a faithful functor and let $A \in \Obj(\C)$. If $f,g \in \Epi_\F(A,-)$, then $f \sim g$ implies that $A/f \cong A/g$.
\end{lemma}
\begin{proof}
The induced morphisms $f/g$ and $g/f$ are inverses of each other. To see this, notice that since $(g/f) \circ f=g$ and $(f/g) \circ g=f$, we have
\[\frac{f}{g} \circ \frac{g}{f} \circ f=\frac{f}{g} \circ g=f.\]
Thus, $(f/g) \circ (g/f)=\id_{A/f}$ since $f$ is an epimorphism. Similarly, we can show that $(g/f) \circ (f/g)=\id_{A/g}$, proving that $A/f \cong A/g$.
\end{proof}

\begin{lemma}\label{lem:cat-ord-preserving}
Let $\F\colon\C \to \D$ be a faithful functor and let $A \in \Obj(\C)$. Suppose that $f,g_1,g_2 \in \Epi_\F(A,-)$ are such that $f \leq g_1$ and $f \leq g_2$. Then:
\renewcommand{\labelenumi}{\normalfont(\roman{enumi})}
\begin{enumerate}
    \item $g_1 \leq g_2$ if and only if $g_1/f \leq g_2/f$;
    \item $g_1 \sim g_2$ if and only if $g_1/f \sim g_2/f$.
\end{enumerate}
\end{lemma}
\begin{proof}
The statement in~(ii) obviously follows from~(i), so we only prove~(i). Suppose that $g_1 \leq g_2$. To show that $g_1/f \leq g_2/f$, we just need to verify that
\[\begin{tikzcd}
A/f \arrow[d,"g_1/f"'] \arrow[r, "g_2/f"] & A/g_2 \\
A/g_1 \arrow[ur,"g_2/g_1"']
\end{tikzcd}\]
commutes. However we have
\[\frac{g_2}{g_1} \circ \frac{g_1}{f} \circ f=\frac{g_2}{g_1} \circ g_1=g_2=\frac{g_2}{f} \circ f,\]
and since $f$ is an epimorphism, our claim follows.

Conversely, suppose that $g_1/f \leq g_2/f$. The diagram
\[\begin{tikzcd}
A \arrow[d,"g_1"'] \arrow[r, "g_2"] & A/g_2 \\
A/g_1 \arrow[ur,"(g_2/f)/(g_1/f)"']
\end{tikzcd}\]
commutes since
\[\frac{(g_2/f)}{(g_1/f)} \circ g_1=\frac{(g_2/f)}{(g_1/f)} \circ (g_1/f) \circ f=\frac{g_2}{f} \circ f=g_2.\]
Thus, we have $g_1 \leq g_2$.
\end{proof}

We can now provide a categorical analogue for the universal algebraic correspondence theorem.

\begin{theorem}[Categorical Correspondence Theorem]\label{thm:cat-correspondence}
Let $\F\colon\C \to \D$ be a faithful functor. Let $A \in \Obj(\C)$ and let $\uparrow[f]_\sim=\{\, [g]_\sim \mid f \leq g \,\}$ be the principal filter of $\Quo_\F(A)$ generated by the quotient $[f]_\sim$. The partially ordered classes $(\uparrow[f]_\sim,\leq)$ and $(\Quo_\F(A/f),\leq)$ are isomorphic.
\end{theorem}
\begin{proof}
Consider the map $[g]_\sim \mapsto [g/f]_\sim$. By Lemma~\ref{lem:cat-ord-preserving}, this map is a strongly order preserving, well-defined injection. To prove that it is surjective, suppose that $h\colon A/f \to B$ is an $\F$-epimorphism, and then show that $h \sim g/f$ for some $g \in\ \uparrow[f]_\sim$. Setting $g=h \circ f$ gives us the desired result.
\end{proof}

To end this section, we present a direct connection between the concept of $\F$-free objects and $\F$-quotients. The following definition for $\F$-free objects generalizes the definition found in~\cite[p.~55]{hungerford} which requires that $\F\colon\C \to \Set$.

\begin{definition}\label{def:free-obj}
Let $\F\colon \C \to \D$ be a faithful functor, $X$ a $\D$-object, $A$ a $\C$-object and $i\colon X \to \F(A)$ a monomorphism. We say the pair $(A,i)$ is \emph{$\F$-free over $X$} if, for any object $B \in \Obj(\C)$ and morphism $f\colon X \to \F(B)$, there exists a unique morphism $\varphi\colon A \to B$ in $\C$ such that the following commutes:
\[\begin{tikzcd}
    X \arrow[hookrightarrow,r,"i"] \arrow[dr,"f"'] & \F(A) \arrow[d,"\F(\varphi)"] \\
    & \F(B)
\end{tikzcd}\]
\end{definition}

\begin{definition}
    With respect to a faithful functor $\F\colon \C \to \D$, a category $\C$ \emph{has $\F$-free objects} if for every $X \in \Obj(\D)$, there is a pair $(A,i)$ which is $\F$-free over $X$.
\end{definition}

Certain types of concrete categories $\C$ arising from algebra always have free objects with respect to their forgetful functors $\F\colon \C \to \Set$. Famously, any non-trivial variety of algebras always has free objects~\cite[p.~170]{cohn}.

\begin{proposition}
    Suppose $\C$ has $\F$-free objects for a faithful functor $\F\colon \C \to \D$. Then, every $\C$-object $K$ is realizable as an $\F$-quotient of an $\F$-free object. That is, we can find a pair $(A,i)$, which is $\F$-free over some $X \in \Obj(\D)$, and an $\F$-epimorphism $\varphi$ such that $K=A/\varphi$.
\end{proposition}

\begin{proof}
Let $(A,i)$ be the free object over $\F(K)$. By definition, $A$ satisfies the universal property mentioned in Definition~\ref{def:free-obj}. Applying this property when $B=K$ and $f=\id_{\F(K)}$, we get the existence of $\varphi\colon A \to K$ such that
\[\begin{tikzcd}
    \F(K) \arrow[hookrightarrow,r,"i"] \arrow[dr,"\id_{\F(K)}"'] & \F(A) \arrow[d,"\F(\varphi)"] \\
    & \F(K)
\end{tikzcd}\]
The fact that $\F(\varphi)$ is an epimorphism follows trivially from the fact that $\id_{\F(K)}$ is an epimorphism, and so we have $K=A/\varphi$.
\end{proof}

\section{Quotients in elementary classes}\label{sec:elem-classes}

In this section, we assume some familiarity with elementary model theory. Let us fix a first-order language $\Ll$. A particularly important concept that will be used throughout this paper is the notion of a strong homomorphism between $\Ll$-structures.

\begin{definition}[{\cite[p.~24]{manzano}}]
Let $\M$ and $\N$ be $\Ll$-structures. A \emph{strong homomorphism} from $\M$ to $\N$ is a map $f\colon\M \to \N$ such that for all $n \in \mathbb{N}$ and $x_1,\dots,x_n \in M$, where $M$ is the universe of $\M$:
\begin{enumerate}
    \item For every $n$-ary function symbol $F \in \Ll$, we have
    \[f(F^\M(x_1,\dots,x_n))= F^\N(f(x_1),\dots,f(x_n));\]
    \item For every $n$-ary relation symbol $R \in \Ll$, we have
    \[R^\M(x_1,\dots,x_n) \iff R^\N(f(x_1),\dots,f(x_n));\]
\end{enumerate}
\end{definition}

The $\Ll$-structures and strong homomorphisms form a concrete category $\LStr$ under the forgetful functor $\F\colon\LStr \to \Set$ mapping every $\Ll$-structure to its universe. Setting $\Ll=\{c\}$ where $c$ is a constant symbol (or equivalently, a $0$-ary function symbol) as an example, we obtain a category isomorphic to the category of pointed sets $\pSet$.

If we also have an $\Ll$-theory $\T$, then the $\T$-models and strong homomorphisms form a full subcategory $\TMdl$ of $\LStr$ whose objects form an \emph{elementary class axiomatized by $\T$}. For example, if $\Ll=\{+,-,0\}$ (where $-$ denotes the unary negation symbol) and $\T$ consists of the abelian group axioms, then $\TMdl$ is the category $\mathbf{Ab}$. On the other hand, if $\T=\emptyset$, then $\TMdl$ and $\LStr$ coincide. We note that in $\TMdl$ and $\LStr$, bijective strong homomorphisms are precisely isomorphisms.

Throughout the rest of this paper, the universe (or underlying set) of an $\Ll$-structure $\M$ will be denoted as $M$. In addition, the universe of the $\F$-quotient $\M/f$ will be denoted as $M/f$.

In~\cite{barrett}, Barrett introduces the idea of a logical quotient in the context of first-order model theory. This idea generalizes the notion of quotients in universal algebra.

\begin{definition}[\cite{barrett}]
Let $\M$ be an $\Ll$-structure. An equivalence relation $\theta$ on $M$ is a \emph{congruence} on $\M$ if for all $n \in \mathbb{N}$ and $(x_i,y_i) \in \theta$, $i \leq n$, we have:
\begin{enumerate}
    \item For every $n$-ary function symbol $F \in \Ll$,
    \[(F^\M(x_1,\dots,x_n),F^\M(y_1,\dots,y_n)) \in \theta;\]
    \item For every $n$-ary relation symbol $R \in \Ll$,
    \[R^\M(x_1,\dots,x_n) \iff R^\M(y_1,\dots,y_n).\]
\end{enumerate}
The set of congruences on $\M$ is denoted as $\Con(\M)$.
\end{definition}

\begin{remark}
This definition of congruence on $\Ll$-structures is not compatible with the notion of (strong) congruence on hyperalgebras (see~\cite{ameri}), where in this case every $n$-ary hyperoperation is considered as an $(n+1)$-ary relation. This definition is instead motivated by the desire to construct the $\F$-quotients in the concrete category $\LStr$ in a natural way. Theorem~\ref{thm:quo-iso-con}, for example, gives a direct link between $\Quo_\F(\M)$ and $\Con(\M)$, where $\F\colon\LStr \to \Set$ is the forgetful functor.
\end{remark}

It is known that the lattice $\Eq(M)$ of equivalence relations on $M$, ordered by inclusion, is complete. Furthermore, for $\theta_i \in \Eq(M)$, $i \in I$, we have
\[\bigwedge_{i \in I} \theta_i = \bigcap_{i \in I} \theta_i\]
and
\[\bigvee_{i \in I} \theta_i = \bigcup \{\, \theta_{i_1} \circ \dots \circ \theta_{i_k} \mid i_1,\dots,i_k \in I \,\}\]
where $\circ$ denotes the composition of binary relations.

\begin{proposition}
For any $\Ll$-structure $\M$, $(\Con(\M),\subseteq)$ is a complete sublattice of $(\Eq(M),\subseteq)$.
\end{proposition}

We can naturally define quotients of first-order structures using congruences:

\begin{definition}[\cite{barrett}]
Let $\theta$ be a congruence on an $\Ll$-structure $\M$. The \emph{quotient of $\M$ by $\theta$} is defined as the $\Ll$-structure $\M/\theta$ such that:
\begin{enumerate}
    \item The universe of $\M/\theta$ is $M/\theta=\{\, [x]_\theta \mid x \in M \,\}$;
    \item For every $n$-ary function symbol $F \in \Ll$, we have
    \[F^{\M/\theta}([x_1]_\theta,\dots,[x_n]_\theta)= [F^\M(x_1,\dots,x_n)]_\theta;\]
    \item For every $n$-ary relation symbol $R \in \Ll$, we have
    \[R^{\M/\theta}([x_1]_\theta,\dots,[x_n]_\theta) \iff R^\M(x_1,\dots,x_n).\]
\end{enumerate}
\end{definition}

The notation $\M/\theta$ for a quotient generated by a congruence and $\M/f$ for an $\F$-quotient can sometimes be in conflict. To avoid confusion, we shall consistently use lowercase Greek letters such as $\theta$ and $\psi$ to denote congruences and lowercase letters such as $f$ and $g$ to denote $\F$-epimorphisms.

Given a non-abelian group $G$, the quotient $G/G$ is abelian. This implies that the elementary class of non-abelian groups is not closed under quotients. However, we can give a sufficient condition for the class of $\T$-models to be closed under quotients by giving some requirements for $\T$ to fulfill.

Recall that an \emph{atomic formula} is an $\Ll$-formula of the form $R(t_1,\dots,t_n)$ or $s=t$, where $s,t,t_1,\dots,t_n$ are $\Ll$-terms. A \emph{literal} is an atomic formula or the negation of one (i.e., of the form $R(t_1,\dots,t_n)$, $\lnot R(t_1,\dots,t_n)$, $s=t$ or $\lnot(s=t)$).

\begin{definition}[\cite{chang}]
An $\Ll$-formula $\varphi(x_1,\dots,x_n)$ is in \emph{prenex conjunctive normal form} (PCNF) if it is of the form
\[Q_1y_1\dots Q_my_m\ (\psi_{11} \vee \dots\vee \psi_{1n_1}) \wedge \dots \wedge (\psi_{k1} \vee \dots\vee \psi_{kn_k})\]
where for every $1 \leq i \leq m$, $Q_i$ is either the $\forall$ or $\exists$ symbol, and each $\psi_{ij}$ is a literal with free variables in $x_1,\dots,x_n,y_1,\dots,y_m$.
\end{definition}

\begin{remark}\label{rmk:anti-ineq}
It should be noted that every $\Ll$-formula is equivalent to one in PCNF (see~\cite{chang}). Moreover, Barrett~\cite{barrett} showed that formulae written in PCNF without literals of the form $\lnot(s=t)$ have their truth preserved under quotients. Thus, if an elementary class can be axiomatized by formulae of this form, then it is closed under quotients. Varieties of algebras, e.g., groups and rings, provide examples of such classes.
\end{remark}

\begin{example}
An MI-monoid~\cite{holcapek} (MI stands for ``Many Identities'') can be defined as the structure $\G$ with universe $G$ and language $\Ll=\{\circ,E\}$, where $\circ$ is a binary operation and $E$ is a unary relation, such that:
\begin{enumerate}
    \item $\forall x \forall y \forall z\ (x \circ y) \circ z=x \circ (y \circ z)$;
    \item $\exists e \forall x\ (\lnot E(e) \vee (x \circ e=x)) \wedge (\lnot E(e) \vee (e \circ x=x))$;
    \item $\forall a \forall b\ (\lnot E(a) \vee \lnot E(b) \vee E(a \circ b))$;
    \item $\forall x \forall a\ (\lnot E(a) \vee (x \circ a=a \circ x))$.
\end{enumerate}
We say that $e$ is a \emph{pseudoidentity} of $\G$ if $E(e)$ is true. Since this axiomatization includes no $\lnot (s=t)$ literals, it follows from the above discussion that the class of MI-monoids is closed under quotients.
\end{example}

\begin{definition}
Let $f\colon\M \to \N$ be a strong homomorphism between $\Ll$-structures. The \emph{kernel} of $f$ is defined as
\[\ker f=\{\, (x,y) \in M^2 \mid f(x)=f(y) \,\}.\]
\end{definition}

\begin{definition}
Let $\M$ be an $\Ll$-structure and let $\theta \in \Con(\M)$. The \emph{quotient map} $\pi_\theta\colon\M \to \M/\theta$ is defined as $\pi_\theta(x)=[x]_\theta$.
\end{definition}

It is straightforward to show that $\ker f$ is a congruence and that $\pi_\theta$ is a surjective strong homomorphism. The existence of $\pi_\theta$ leads to the fact that every congruence on $\M$ is the kernel of some surjective strong homomorphism from $\M$.

\begin{proposition}\label{prop:theta=kerpi}
If $\M$ is an $\Ll$-structure and $\theta$ is a congruence on $\M$, then $\theta=\ker \pi_\theta$.
\end{proposition}

We now connect the notion of $\F$-quotients from the Section~\ref{sec:F-quot} and the notion of quotients of $\Ll$-structures from this section.

\begin{lemma}\label{lem:ord-preserving}
Let $\M$ be an $\Ll$-structure and let $\F\colon\LStr \to \Set$ be the forgetful functor. Suppose that $f,g \in \Epi_\F(\M,-)$. Then:
\renewcommand{\labelenumi}{\normalfont(\roman{enumi})}
\begin{enumerate}
    \item $f \leq g$ if and only if $\ker f \subseteq \ker g$;
    \item $f \sim g$ if and only if $\ker f=\ker g$.
\end{enumerate}
\end{lemma}
\begin{proof}
Since~(ii) is easily obtained from~(i), we only give a proof of~(i). Suppose that $f \leq g$ and that $(x,y) \in \ker f$. Since $f(x)=f(y)$, we then have
\[g(x)=((g/f) \circ f)(x)=((g/f) \circ f)(y)=g(y),\]
and thus $(x,y) \in \ker g$. Conversely, suppose that $\ker f \subseteq \ker g$. Set $h\colon\M/f \to \M/g$ to be the map $h(y)=g(x)$, where $x$ is any element of $M$ such that $f(x)=y$. From the fact that $f$ is surjective and that $\ker f \subseteq \ker g$, it is clear that $h$ is well-defined. Moreover, we can easily see that $g=h \circ f$. Now let $y_1,\dots,y_n \in M/f$ and $F,R \in \Ll$ be $n$-ary. Suppose that $f(x_i)=y_i$ for $i \leq n$. The map $h$ preserves the function $F$ since
\begin{align*}
h(F^{\M/f}(y_1,\dots,y_n))&=h(F^{\M/f}(f(x_1),\dots,f(x_n)))\\&=(h \circ f)(F^\M(x_1,\dots,x_n))\\&=g(F^\M(x_1,\dots,x_n))\\&=F^{\M/g}(g(x_1),\dots,g(x_n))\\&=F^{\M/g}(h(y_1),\dots,h(y_n)).
\end{align*}
Moreover,
\begin{align*}
R^{\M/f}(y_1,\dots,y_n) &\iff R^{\M/f}(f(x_1),\dots,f(x_n)) \\&\iff R^\M(x_1,\dots,x_n) \\&\iff R^{\M/g}(g(x_1),\dots,g(x_n)) \\&\iff R^{\M/g}(h(y_1),\dots,h(y_n)),
\end{align*}
which implies that $h$ also preserves the relation $R$. Thus, $h$ is a strong homomorphism and so we have $f \leq g$.
\end{proof}

\begin{theorem}\label{thm:quo-iso-con}
Let $\M$ be an $\Ll$-structure and let $\F\colon\LStr \to \Set$ be the forgetful functor. The map $(\Quo_\F(\M),\leq) \to (\Con(\M),\subseteq)$, $[f]_\sim \mapsto \ker f$ defines a lattice isomorphism.
\end{theorem}
\begin{proof}
This map is surjective by Proposition~\ref{prop:theta=kerpi}. The theorem then clearly follows from Lemma~\ref{lem:ord-preserving}, which ensures that the map is a strongly order-preserving, well-defined injection.
\end{proof}

\begin{proposition}\label{prop:M/f iso M/ker f}
Let $\M$ be an $\Ll$-structure and let $\F\colon\LStr \to \Set$ be the forgetful functor. Let $f \in \Epi_\F(\M,-)$. Then the map $\M/\ker f \to \M/f$ defined by $[x]_{\ker f} \mapsto f(x)$, where $x \in M$, is a well-defined isomorphism.
\end{proposition}
\begin{proof}
It is clear that the map is a well-defined bijection. Let $x_1,\dots,x_n \in M$ and $F,R \in \Ll$ be $n$-ary. We see that the map preserves $F$ by combining the fact that
\[F^{\M/\ker f}([x_1]_{\ker f},\dots,[x_n]_{\ker f})=[F^\M(x_1,\dots,x_n)]_{\ker f}\]
and
\[f(F^\M(x_1,\dots,x_n))=F^{\M/f}(f(x_1),\dots,f(x_n)).\]
The map also preserves $R$ since
\begin{align*}
R^{\M/\ker f}([x_1]_{\ker f},\dots,[x_n]_{\ker f}) &\iff R^\M(x_1,\dots,x_n) \\ &\iff R^{\M/f}(f(x_1),\dots,f(x_n)).
\end{align*}
Therefore, the map is a strong homomorphism.
\end{proof}

\section{Model-theoretic isomorphism theorems}

In this section, the isomorphism theorems for universal algebras (see~\cite{burris}) are generalized to the setting of $\F$-quotients. Many of the isomorphism theorems presented here are simpler to show due to the progress made previously in Sections 2 and 3. Fix $\Ll$-structures $\M$, $\N$. It can be verified that the image $f(\M)$ of a strong homomorphism $f\colon\M \to \N$ forms a substructure of $\N$.

\begin{theorem}[First Isomorphism Theorem]
If $f\colon\M \to \N$ is a strong homomorphism, then
\[\frac{\M}{\ker f} \cong f(\M).\]
\end{theorem}
\begin{proof}
Set a surjective homomorphism $f'\colon\M \to f(\M)$ as $f'(x)=f(x)$. By Proposition~\ref{prop:M/f iso M/ker f}, we have
\[\frac{\M}{\ker f}=\frac{\M}{\ker f'} \cong \frac{\M}{f'}=f(\M)\]
as desired.
\end{proof}

Suppose $K$ is an elementary class axiomatized by PCNF formulae without $\lnot (s=t)$ literals. By Remark~\ref{rmk:anti-ineq}, for any $\M,\N \in K$ and any strong homomorphism $f\colon \M \to \N$, $f(\M) \in K$ also, since it is isomorphic to $\M/\ker f$. In other words, the first isomorphism theorem easily shows that $K$ is closed under strong homomorphic images.

Suppose that $\N$ is a substructure of $\M$, and that $\theta$ is a congruence on $\M$. We can define a subset of $M$ as $N^\theta=\{\, x \in M \mid N \cap [x]_\theta \neq \emptyset \,\}$. The smallest substructure of $\M$ containing the set $N^\theta$ is denoted as $\N^\theta$.

\begin{proposition}
If $\N$ is a substructure of $\M$ and $\theta \in \Con(\M)$, then the universe of $\N^\theta$ is $N^\theta$.
\end{proposition}

In addition, we can define the restriction of $\theta$ to a subuniverse $N$ as $\theta|_N=\theta \cap N^2$. It is not hard to show that $\theta|_N \in \Con(\N)$.

\begin{theorem}[Second Isomorphism Theorem]
If $\N$ is a substructure of $\M$ and $\theta \in \Con(\M)$, then
\[\frac{\N}{\theta|_N} \cong \frac{\N^\theta}{\theta|_{N^\theta}}.\]
\end{theorem}
\begin{proof}
It is straightforward to show that the map $[y]_{\theta|_N} \mapsto [y]_{\theta|_{N^\theta}}$, $y \in N$ is a well-defined isomorphism.
\end{proof}

\begin{definition}
If $\theta \subseteq \psi$ are congruences on $\M$, then
\[\psi/\theta=\{\, ([x]_\theta,[y]_\theta) \in (M/\theta)^2 \mid (x,y) \in \psi \,\}.\]
\end{definition}

It is straightforward to confirm that $\psi/\theta \in \Con(\M/\theta)$. From the following definition, we find that $\ker(g/f)$, where $f \leq g$ are surjective strong homomorphisms, and $\ker g/\ker f$ are essentially equivalent.

\begin{proposition}\label{prop:fracker-equiv-kerfrac}
Let $\F\colon\LStr \to \Set$ be the forgetful functor and let $f,g \in \Epi_\F(\M,-)$. Suppose $\varphi\colon\M/\ker f \to \M/f$ is the isomorphism $[x]_{\ker f} \mapsto f(x)$. If $f \leq g$, then
\[\varphi^2(\ker g/\ker f)=\ker(g/f).\]
\end{proposition}
\begin{proof}
For every $x,y \in M$, we have the following chain of equivalences:
\begin{align*}
([x]_{\ker f},[y]_{\ker f}) \in \ker g/\ker f &\iff (x,y) \in \ker g\\
&\iff g(x)=g(y)\\
&\iff ((g/f) \circ f)(x)=((g/f) \circ f)(y)\\
&\iff (f(x),f(y)) \in \ker(g/f).
\end{align*}
This proves that $\varphi^2(\ker g/\ker f)=\ker(g/f)$.
\end{proof}

\begin{theorem}[Third Isomorphism Theorem]
If $\theta \subseteq \psi$ are congruences on $\M$, then
\[\frac{(\M/\theta)}{(\psi/\theta)} \cong \frac{\M}{\psi}.\]
\end{theorem}
\begin{proof}
Take $\F\colon\LStr \to \Set$ to be the forgetful functor and let $f,g \in \Epi_\F(\M,-)$, $f \leq g$, be such that $\psi=\ker g$ and $\theta=\ker f$. By Proposition~\ref{prop:fracker-equiv-kerfrac}, we have
\[\frac{(\M/\theta)}{(\psi/\theta)}=\frac{(\M/\ker f)}{(\ker g/\ker f)} \cong \frac{\varphi(\M/\ker f)}{\varphi^2(\ker g/\ker f)} = \frac{(\M/f)}{\ker (g/f)}.\]
Moreover, by Proposition~\ref{prop:M/f iso M/ker f} and Proposition~\ref{prop:cat-third-iso}, we obtain
\[\frac{(\M/f)}{\ker (g/f)} \cong \frac{(\M/f)}{(g/f)}=\frac{\M}{g} \cong \frac{\M}{\ker g}=\frac{\M}{\psi},\]
which completes the proof.
\end{proof}

\begin{theorem}[Correspondence Theorem]
Let $[\theta,M^2]=\{\, \psi \mid \theta \subseteq \psi \,\}$ be the principal filter of $\Con(\M)$ generated by the congruence $\theta$. We then have that $[\theta,M^2]$ and $\Con(\M/\theta)$ are isomorphic lattices.
\end{theorem}
\begin{proof}
Let $\F\colon\LStr \to \Set$ be the forgetful functor. By invoking Theorem~\ref{thm:quo-iso-con}, Proposition~\ref{prop:M/f iso M/ker f} and Theorem~\ref{thm:cat-correspondence} judiciously, we have that
\[[\theta,M^2] \cong\ \uparrow [f]_\sim \cong \Quo_\F(\M/f) \cong \Con(\M/f) \cong \Con(\M/\theta),\]
where $f \in \Epi_\F(\M,-)$ is such that $\theta=\ker f$.
\end{proof}

The proofs of the model-theoretic first and third isomorphism theorems, as well as the correspondence theorem, all used $\F$-quotients to provide a more concise argument. Mainly, we utilized Theorem~\ref{thm:quo-iso-con} and Proposition~\ref{prop:M/f iso M/ker f} as a means for us to invoke the results obtained previously in Section~\ref{sec:F-quot} (e.g., Proposition~\ref{prop:cat-third-iso} and Theorem~\ref{thm:cat-correspondence}).

This paper has explored a few applications of functor-based categorical quotients to congruence properties of first-order structures. While the scope of this paper is limited to isomorphism theorems, one could as well provide other generalizations, e.g., for the Zassenhaus lemma and the Jordan-Hölder theorem. Alternatively, one could attempt to define the notion of ``weak'' congruences which agrees with the $\F$-quotients on the concrete category of $\Ll$-structures with weak homomorphisms as morphisms.

\end{document}